\documentclass[12pt]{article}
\usepackage{amsfonts}
\usepackage[pctex32]{graphics}
\usepackage[dvips]{graphicx}
\usepackage{amsthm}
\usepackage{amsmath}
\usepackage{amssymb}
\usepackage{fontenc}
\usepackage[all]{xy}
\usepackage{latexsym}
\usepackage{layout}
\usepackage{epsfig}
\usepackage{graphicx}
\usepackage{pstricks,pst-node,pst-plot} 
\usepackage{color} 
\title{Shattering Thresholds for Random Systems of Sets, Words, and Permutations}
\author {Anant P.~Godbole\\
Department of Mathematics and Statistics\\
East Tennessee State University\and Samantha Pinella\\
School of Mathematics\\
University of Edinburgh\and Yan Zhuang\\
Department of Mathematics\\
Brandeis University}
\begin{document}
\def\qed{\vbox{\hrule\hbox{\vrule\kern3pt\vbox{\kern6pt}\kern3pt\vrule}\hrule}}
\def\ms{\medskip}
\def\n{\noindent}
\def\ep{\varepsilon}
\def\G{\Gamma}
\def\D{\Delta}
\def\lr{\left(}
\def\ls{\left[}
\def\rs{\right]}
\def\lf{\lfloor}
\def\rf{\rfloor}
\def\lg{{\rm lg}}
\def\lc{\left\{}
\def\rc{\right\}}
\def\rr{\right)}
\def\ph{\varphi}
\def\p{\mathbb P}
\def\nk{n \choose k}
\def\a{\cal A}
\def\s{\cal S}
\def\e{\mathbb E}
\def\l{\lambda}
\def\va{\mathbb V}
\newcommand{\bsno}{\bigskip\noindent}
\newcommand{\msno}{\medskip\noindent}
\newcommand{\oM}{M}
\newcommand{\omni}{\omega(k,a)}
\newtheorem{thm}{Theorem}[section]
\newtheorem{con}{Conjecture}[section]
\newtheorem{claim}[thm]{Claim}
\newtheorem{deff}[thm]{Definition}
\newtheorem{lem}[thm]{Lemma}
\newtheorem{cor}[thm]{Corollary}
\newtheorem{rem}[thm]{Remark}
\newtheorem{prp}[thm]{Proposition}
\newtheorem{ex}[thm]{Example}
\newtheorem{eq}[thm]{equation}
\newtheorem{que}{Problem}[section]
\newtheorem{ques}[thm]{Question}
\providecommand{\floor}[1]{\left\lfloor#1\right\rfloor}
\maketitle
\begin{abstract}
This paper considers a problem that relates to the theories of covering arrays \cite{co}, permutation patterns \cite{kn}, Vapnik-\v Cervonenkis (VC) classes \cite{bl},\cite{du}, and probability thresholds \cite{as}.  Specifically, we want to find the number of  subsets of $[n]:=\{1,2,\ldots,n\}$ we need to randomly select, in a certain probability space, so as to respectively ``shatter" all $t$-subsets of $[n]$.  Moving from subsets to words, we ask for the number of $n$-letter words on a $q$-letter alphabet that are needed to shatter all $t$-subwords of the $q^n$ words of length $n$.  Finally, we explore the number of random permutations of $[n]$ needed to shatter (specializing to $t=3$), all length 3 permutation patterns in specified positions.  We uncover a very sharp zero-one probability threshold for the emergence of such shattering; Talagrand's isoperimetric inequality in product spaces \cite{as} is used as a key tool, along with the second moment method. 
\end{abstract}
\section{Introduction}  In this section, we give the necessary background on covering arrays, Vapnik-\v Cervonenkis classes, and permutation patterns, and then explain our goals.  A $k\times n$ array with entries from the alphabet $\{0,1,\ldots, q-1\}$ is said to be a $(t,q,n,k,\l)$-covering array, or briefly a $t$-covering array, if for each of the ${n\choose t}$ choices of $t$ columns, each of the $q^t$ $q$-ary words of length $t$ can be found at least $\l$ times among the rows of the selected columns.  Covering arrays are used as valuable tools in software testing; see, e.g., \cite{co}, which is a comprehensive survey of the theory of $t$-covering arrays.  In this paper, we will focus solely on the case $\l=1$.  If $q=2$, we can interpret any row as the characteristic vector of a subset of $[n]$ -- by making a correspondence between the positions where the row has ones, and the set of those positions.  We thus  have the following alternative formulation of covering arrays: A family ${\cal F}$ of subsets of $[n]$ is a $t$-covering array if for each $\{a_1,\ldots,a_t\}\subset[n]$,
\[\vert\{\{a_1,\ldots,a_t\}\cap F\}:F\in{\cal F}\vert= 2^t.\]
We next see how this definition relates to that of VC classes.

A  class ${\cal F}$ of subsets of an abstract set ${\cal Y}$ is said to {\it shatter} a subset $A=\{a_1,\ldots,a_t\}\subseteq {\cal Y}$ if
\[\vert\{A\cap F\}:F\in{\cal F}\vert= 2^t.\] Furthermore, the VC dimension ${\cal VC(F)}$ of ${\cal F}$ \cite{du} is the cardinality of the smallest {\it unshattered} set (the dimension is $\infty$ if all sets of all finite size are shattered.)  A class ${\cal F}$ is said to be a VC class if ${\cal VC(F)}<\infty$.  Many canonical examples of VC classes are driven by underlying geometric considerations.  For example, consider the {\it infinite} family ${\cal F}$ of subsets of ${\cal Y}={\mathbb R}$ of the form $(-\infty,x]:x\in{\mathbb R}$.  Every set of size 1 is clearly shattered by ${\cal F}$.  Next letting $t=2$, we consider the class of all 2-element subsets of ${\mathbb R}$.  It is then clear that
\[\vert\{\{a_1,a_2; a_1<a_2\}\cap F\}:F\in{\cal F}\vert= 3<2^2,\]
since it is impossible for $F\in{\cal F}$ to intersect the two element set $\{a_1,a_2\}$ in its larger element $a_2$.  It follows that ${\cal VC(F)}=2$.  To give another example, if ${\cal F}$ consists of all convex sets in ${\cal Y}={\mathbb R}^2$, then it is impossible for elements of ${\cal F}$ to ``shatter" a three element subset $A=\{a_1,a_2,a_3\}$ of collinear points, and thus
\[\vert\{A\cap F:F\in{\cal F}\}\vert= 7<2^3.\]  
Since every 2-element set can be shattered by convex sets, we have that ${\cal VC(F)}=3$ in this case.  

VC classes were first defined and used in the context of uniform limit theorems in statistics \cite{du}, \cite{vc}; later, their use was extended to learning theory \cite{bl}, \cite{va}. The alternative (and perhaps more popular) definition of the VC dimension of ${\cal F}$ is ``the cardinality of the largest shattered set".  In many cases, e.g., in the first example given above, the largest shattered set and the smallest unshattered set differ in size by 1; in general, however, this is not the case, as in the second example.  The reason why we use the first definition of the VC dimension is given later in this section.

The above discussion reveals that with ${\cal Y}=[n]$, n system ${\cal F}$ of finite subsets of $[n]$, arranged in a rectangular array,  is a binary $t$-covering array if and only if ${\cal VC(F)}\ge t+1$.  An explanation follows:  If ${\cal F}$ is $t$-covering, then for each set $A$ of size $t$ and each $B\subseteq A$, there exists $F\in{\cal F}$ such that $F\cap A=B$; thus every set of size $t$ is shattered, and the smallest unshattered set must be of size $t+1$ or more.  The reverse argument is valid too.

When $q\ge 3$, covering arrays are often described in terms of words.  We will use this terminology in this paper too, but the notions of shattering and VC dimension are probably best described using the language of multisets.  We will interpret a $k\times n$ array $\{a_{ij}\}_{1\le i\le k; 1\le j\le n}$, with entries from $\{0,1,\ldots,q-1\}$ as consisting of $k$ multisets, with the $i$th multiset containing the element $j$ $a_{ij}$ times, where the {\it degree} of the multiset, i.e., the maximum number of times an element may appear in it, is bounded by $q-1$.  The notion of the intersection of two multisets $A,B$ is defined in the natural way.  For example, $\{1,1,2,2,3\}\cap\{1,1,1,1,2,3,3\}=\{1,1,2,3\}$.  We say that a collection ${\cal F}$ of $k$ multisets shatters a multiset $A$ with $t$ distinct elements each repeated $q-1$ times, if 
\[\vert\{A\cap F:F\in{\cal F}\}\vert= q^t.\] As before the VC dimension ${\cal VC(F)}$ of ${\cal F}$ is the cardinality of the smallest {\it unshattered} multiset of the above type, with $q$ fixed and the minimum taken over $t$ (${\cal VC(F)}=\infty$ if there is no such smallest $t$).  We thus see that ${\cal F}$ is a $(t,q,n,k,1)$-covering array if and only if ${\cal VC(F)}\ge t+1$.

We next turn to permutations.  The theory of permutation patterns was initiated by Knuth \cite{kn}, and continues to be an area of active investigation.  We say that a  permutation $\pi\in S_n$ {\it contains} the permutation $\rho\in S_t$ if there exist indices $1\le i_1<i_2<\ldots<i_t\le n$ such that $(\pi_{i_1},\ldots,\pi_{i_t})$ and $(\rho_1,\ldots,\rho_t)$ are order isomorphic; if not we say that $\pi$ {\it avoids} $\rho$.  Enumeration questions are critical in this area.  For example it is known that for $t=3$, the number of $(i,j,k)$ avoiding $n$-permutations is given by the Catalan numbers ${{2n}\choose{n}}/(n+1)$ for each of the six choices of $i,j,k$ \cite{bo}.  Moreover, the Stanley-Wilf conjecture, namely that for fixed $\rho$, the number of $\rho$-avoiding $n$-permutations is asymptotic to $C^n$ for some $1\le C=C_\rho<\infty$, was recently proved by Marcus and Tardos \cite{mt}. How might shattering and VC dimension be defined in the case where ${\cal F}$ consists of an array of $k$ $n$-permutations $(\pi_1,\ldots,\pi_k)$?  Using the language of covering arrays, we shall say that the VC dimension is at least $t+1$ if for each choice of $t$ columns and $\rho\in S_t$, at least one row of the selected columns contains entries order isomorphic to those of $\rho$.  This is equivalent to saying that the $k$ permutations, restricted to any $t$ positions, shatter all the $t!$ permutations on those positions.

The {\it size} of a $t$-covering (or other) rectangular array will refer to the number of rows it contains, expressed as a function of the number of columns.  Research on $t$-covering arrays has focused on finding arrays of small size.  In \cite{sl}, for example, the case of $t=3$ is studied in detail, and Roux's result that there exist $3$-covering binary arrays of size $7.5\lg n$ is proved, where $\lg=\log_2$.  This result was re-proved in \cite{gss} using the Lov\'asz Local Lemma (see \cite{as}), where the underlying probability model consisted, as in the work of Roux, of independently placing an equal number of ones and zeros in each column.  This model is intractable for general values of $t$ and $q$; accordingly, the general upper bound on the size of covering arrays was proved in \cite{gss} by reverting to a simple multinomial model, where each spot in the $k\times n$ array is independently and uniformly chosen from the set $\{0,1,\ldots,q-1\}$.  However the Lov\'asz Lemma is an existence result whose conclusion is that there is a positive probability that there are no ``bad events,"  i.e., that 
\[k\ge K\Rightarrow\p({\rm array\ is\ }t-{\rm covering})>0,\]
so that a $t$-covering array with $K$ rows exists.  By contrast, in this paper we are looking for results, still in the $\log n$ domain, that are of the form
\[k\le k_0(n)\Rightarrow \p({\rm array\ is\ }t-{\rm covering})\to 0\enspace(n\to\infty),\]
\[k\ge k_1(n)\Rightarrow \p({\rm array\ is\ }t-{\rm covering})\to 1\enspace(n\to\infty),\]
and where the gap $[k_0(n),k_1(n)]$ is not too wide.  We will use the simple first moment method (linearity of expectation) together with Talagrand's isoperimetric inequalities, to establish such a result in Section 2.

The situation is a little more nuanced when we turn to the question of shattering permutations.  First of all, we are only able to prove clean results when $t=3$, but, more importantly, it is also meaningful to consider {\it large} arrays with small VC dimension.  For example, if we wrote each of the $\sim 4^n$ 123-avoiding $n$ permutations in a rectangular array, there would be no triple order isomorphic to 123 in {\it any} set of 3 columns, and the VC dimension of this array would be 2, using the ``largest shattered set" definition.  The relevant question would be to investigate how much better than that we could do while still maintaining the VC dimension.  This is the approach taken by Cibulka and Kyn\v{c}l \cite{ck}, who give {\it superexponential} bounds on the size of the extremal such array for $t=3$.  Our motivation, using the ``smallest unshattered set" definition of VC dimension, is as follows:  As with words, we want to investigate whether for every $t$, as $n$ tends to infinity, there is an interval of values of $[k_0,k_1]=[k_0(n),k_1(n)], k_i(n)=\Theta(\lg n), i=0,1,$ such that an array of $k\le k_0$ random permutations has VC dimension $\ge t+1$ with probability that tends to 0 as $n\to\infty$, and such that an array with more than $k_1$ rows has VC dimension  $\ge t+1$ almost surely.  Results along these lines are proved in Section 3, and it is their probabilistic nature that make the use of the ``smallest unshattered set" definition appropriate.  Consider $k$ of order $\lg n$.  Assuming for simplicity that $q=2$, for each $t$ we have many more rows than possible words, and a simple ``balls in boxes" argument reveals that for some choice of $t$ columns, each binary word will be present in some row of the selected columns with high probability.  There will thus be many sets of large size $t$ that {\it are shattered} and the ``largest shattered set" VC dimension would be quite large.  The point is that there will be a few unshattered sets of small size, and it is those that we wish to understand.

As pointed out by one of the referees of this paper, 
sets of permutations with VC dimension $t+1$ are exactly the $t$-scrambling
permutations from the papers \cite{spencer}, \cite{furedi}, \cite{radhakrishnan} and \cite {tarui}.  In fact, Spencer's paper \cite{spencer} uses the same method as in Theorem 3.1 to give a general upper bound on the size of $t$-scrambling permutations, and the other papers focus on improving upper and lower bounds for the minimum size of scrambling permutations.

\section{Shattering Subsets and Words}  We use the following model. Let ${\cal F}$ be a randomly generated rectangular array of
$k$ words, each of length $n$ and obtained by selecting each position in the $k\times n$ array to independently and uniformly be one of the letters of the  ``alphabet" $\{0,1,\ldots, q-1\}$.  Denote the words in ${\cal F}$ as $F_{1},\ldots,F_{k}$. 
As noted in Section 1, if $q=2$, then ${\cal F}$ is simply a {\it random system} of $k$ subsets
of $[n]$. We will use \textit{rows} to refer to the words in ${\cal F}$ and
\textit{columns} to refer to the character positions.  
In this section, we show that the threshold, under our model, for the property
``${\cal F}$ shatters all $t$-words'' (which is an alternative term we use for multiset shattering) occurs at the level $\frac{t}{\lg\left({q^{t}}/({q^{t}-1})\right)}\lg\left(n\right)$; this will allow us to determine with high probability
the VC dimension of a random word array.
Deriving an upper threshold
is easy:
\begin{thm}
Let $q\ge 2$, $n\ge t\ge 1$, $k\ge\frac{t\lg n}{\lg\left({q^{t}}/({q^{t}-1})\right)}(1+o(1))$,
and let ${\cal F}$ be a randomly generated array of $k$ words.  Then all $t$-words are shattered almost surely by ${\cal F}$, i.e., the probability that ${\cal F}$ is a covering array tends to 1 as $n\to\infty$.\end{thm}
\begin{proof}
Let $X$ be the number of sets of $t$ columns corresponding to unshattered
$t$-words. By Markov's inequality $\p(X\ge a)\le \e(X)/a$ [valid  for non-negative random variables $X$], and linearity of expectation, we have:
\[
\p\left(X\geq1\right)\leq\e\left(X\right)\leq{n \choose t}q^{t}\left(\frac{q^{t}-1}{q^{t}}\right)^{k}\leq\frac{n^{t}}{t!}q^{t}\left(\frac{q^{t}-1}{q^{t}}\right)^{k}\to0
\]
provided that 
\[{k\ge\frac{t\lg n+\omega_n(1)-\lg t!+t\lg q}{\lg\left({q^{t}}/({q^{t}-1})\right)}=\frac{t\lg n}{\lg\left({q^{t}}/({q^{t}-1})\right)}(1+o(1)):=k_1(n),}\]
where we have used the standard notation $\omega_n(1)$ for a function growing to infinity arbitrarily slowly.  This proves the result.\hfill
\end{proof}

Proving that the lower threshold function $k_0(n)$ is of the same magnitude is tantamount to showing that the random variable $X$ is sharply concentrated around its mean. In some sense this was done in \cite{gss}, but using a na\"ive (and ultimately incorrect) probability model that was only shown to be valid for $t=3$.  To give a more rigorous proof in this paper, we shall apply Talagrand's inequality in the form found in \cite{as}.
This inequality is applicable for random variables that
are \textit{1-Lipschitz}:
\begin{deff}
Let $Z$ be a random variable expressed as a function of $N$ independent
variables $\{Z_{i}\}_{i=1}^N$. We call $Z$ \textit{1-Lipschitz} if $$\left|Z\left(Z_{1},\ldots,Z_{N}\right)-Z\left(Z_{1}^{*},\ldots,Z_{N}^{*}\right)\right|\leq1$$
whenever $Z_{i}\ne Z_{i}^{*}$ for at most one $i$.
\end{deff}
The random variable $X$, counting the number of ``defective" $t$-tuples of columns (i.e. those sets
corresponding to unshattered $t$-words), depends on $nk$ mutually independent random variables. It is not, however 1-Lipschitz since an added
presence or absence of a specific character may change $X$ by more
than 1 due to overlapping columns. However, if we define $Y$ as the maximum number of {\it non-overlapping} sets
of ``defective" columns, then $Y$ is 1-Lipschitz. 

Talagrand's inequality also involves the notion of a {\it certification
function}:
\begin{deff}
Let $Z$ be a random variable expressed as a function of $N$ independent
variables $Z_{i}$, and let $f:\mathbb{N\rightarrow\mathbb{N}}$ be
a function. We call \textit{$Z$ $f$-certifiable} if for every $s\ge 1$, $Z\ge s$ can
always be verified to be true by some $f(s)$-tuple  of the $N$ independent random
variables. In this case, we call $f$ a \textit{certification function}
for $Z$.
\end{deff}
Given the random variable $Y$ as above, to verify that there are at least $s$ non-overlapping
sets of unshattered $k$-words, it is easy to see that it suffices to know $kts$ of the entries in the array.  Thus $f$ is linear and $f(s)=kts$.

Talagrand's inequality is reproduced below for completeness:
\begin{thm}[\textbf{Talagrand's Inequality}]
\textbf{\textup{}}Let $Z$ be a 1-Lipschitz random variable with certification
function $f$. Then, for all $m,u>0$:

\[
\p(Z\leq m-u\sqrt{f\left(m\right)})\p\left(Z\geq m\right)\leq e^{-\frac{u^{2}}{4}}
\]

\end{thm}
Applying Talagrand's inequality to the variable $Y$ with $m={\rm median}(Y)$ (so that $\p(Y\ge m)\ge 1/2$) and $u={\sqrt{\frac{m}{kt}}}$, we see that
\begin{equation}
\p(Y=0)\le 2e^{-\frac{m}{4kt}}.\end{equation}  We will use (1) in an appropriate way to get the lower threshold; specifically, we need to derive conditions under which (i) $\e(Y)$ and $\e(X)$ are close; and (ii) $\e(Y)$ and $m$ are close. A series of technical lemmas that lead to (i) and (ii) are presented next.
\begin{lem}
Let $\Gamma$ and $\Delta$ be distinct non-disjoint sets of $t$ columns.
Let $r$ be the number of overlapping elements of $\Gamma$ and $\Delta$;
i.e., $r=\left|\Gamma\cap\Delta\right|$. Define the indicator random
variable $I_{\Gamma}$ as being $1$ if $\Gamma$ is missing at least one $t$-word
and 0 otherwise. Then

\[
\p\left(I_{\Gamma}I_{\Delta}=1\right)\leq q^{2t-r}\left(\frac{q^{t}+q^{r-t}-2}{q^{t}}\right)^{k}\left\{ 1+o\left(1\right)\right\}\enspace(k\to\infty). 
\]
\end{lem}

\begin{proof}
Lemma 2.5 generalizes a result in \cite{gss} and this proof is similar. Let $A_{m}$ be the event that exactly $m$ words are
missing from $\Gamma$. We have

\begin{align}
\p\left(I_{\Gamma}I_{\Delta}=1\right) &=\p\left(I_{\Gamma}=1\right)\p\left(I_{\Delta}=1|I_{\Gamma}=1\right)\nonumber\\
 & =\p\left(I_{\Gamma}=1\right)\p\left(I_{\Delta}=1|A_{1}\cup A_{2}\cup\dotsb\cup A_{q^{t}-1}\right)\nonumber\\
 & =\p\left(I_{\Gamma}=1\right)\left[\frac{\p\left(I_{\Delta}=1\cap A_{1}\right)+\dotsb+\p\left(I_{\Delta}=1\cap A_{q^{t}-1}\right)}{\p\left(A_{1}\cup A_{2}\cup\dotsc\cup A_{q^{t}-1})\right)}\right]\nonumber\\
 & \leq\p\left(I_{\Gamma}=1\right)\left[\p\left(I_{\Delta}=1|A_{1}\right)+\left(\frac{\p\left(A_{2}\right)+\dotsb+\p\left(A_{q^{t}-1}\right)}{\p\left(A_{1}\cup\ldots\cup A_{q^t-1}\right)}\right)\right]\nonumber\\
 & \leq\p\left(I_{\Gamma}=1\right)\left[\p\left(I_{\Delta}=1|A_{1}\right)+\frac{\binom{q^{t}}{2}\left(\frac{q^{t}-2}{q^{t}}\right)^{k}}{q^{t}\left(\frac{q^{t}-1}{q^{t}}\right)^{k}-\binom{q^{t}}{2}\left(\frac{q^{t}-2}{q^{t}}\right)^{k}}\right]\nonumber\\
 & =\p\left(I_{\Gamma}=1\right)\cdot\nonumber\\
&\qquad\left[\p\left(I_{\Delta}=1|A_{1}\right)+\frac{\binom{q^{t}}{2}\left(\frac{q^{t}-2}{q^{t}}\right)^{k}}{q^{t}\left(\frac{q^{t}-1}{q^{t}}\right)^{k}\left(1-\left(\frac{q^{t}-1}{2}\right)\left(\frac{q^{t}-2}{q^{t}-1}\right)^{k}\right)}\right]\nonumber\\
 & =\p\left(I_{\Gamma}=1\right)\cdot\nonumber\\
&\qquad\left[\p\left(I_{\Delta}=1|A_{1}\right)+\left(\frac{q^{t}-1}{2}\right)\left(\frac{q^{t}-2}{q^{t}-1}\right)^{k}\left\{ 1+o\left(1\right)\right\} \right].
\end{align}
Since $\p\left(I_{\Gamma}=1\right)\leq q^{t}\left(1-q^{-t}\right)^{k},$
the problem reduces to upper-bounding $\p\left(I_{\Delta}=1|A_{1}\right)$.
Exactly one word is missing in $\Gamma$; let us denote that word
by $\gamma$. Assume, without loss of generality, that the first $r$ columns of $\Delta$ are the same as the last $r$ columns of $\Gamma$.  We consider two cases. Let $p_{1}$ be the conditional probability
that a word beginning with the last $r$ characters of $\gamma$ is
also missing in $\Delta$; there are $q^{t-r}$ such words. Let $p_{2}$
be the probability that a word not beginning with these same $r$
characters is missing in $\Delta$; there are $q^{t}-q^{t-r}$ such
words. Hence

\[
\p\left(I_{\Delta}=1|A_{1}\right)\le q^{t-r}p_{1}+\left(q^{t}-q^{t-r}\right)p_{2}.
\]

We first calculate $p_{1}$. We know that $\gamma$ is the only word
missing in $\Gamma$, so for each of the remaining $q^{t}-1$ words in this first category,
there is at least one   row in $\Gamma$ containing that word. Take away one
such row for each of these $q^{t}-1$ words; each of the other rows
are randomly assigned to one of these $q^{t}-1$ words with probability
$\frac{1}{q^{t}-1}$ each. This process enables one to realize the probability distribution of the content of the rows of $\Gamma$ given that $A_1$ has occurred.  Let $\mathcal{A}$ be the number of rows
in $\Delta$ that coincide with those of $\Gamma$ in the overlapping $r$ positions; note that 
$\mathcal{A}$ is at least $q^{t-r}-1$. Then for $a\ge0$,

\[
\p(\mathcal{A}=a+q^{t-r}-1)=\binom{k-\left(q^{t}-1\right)}{a}\left(\frac{q^{t-r}-1}{q^{t}-1}\right)^{a}\left(\frac{q^{t}-q^{t-r}}{q^{t}-1}\right)^{k-\left(q^{t}-1\right)-a}.
\]
Using the binomial theorem
we obtain

\begin{align*}
p_{1} & =\sum_{a=0}^{k-\left(q^{t}-1\right)}\binom{k-\left(q^{t}-1\right)}{a}\left(\frac{q^{t-r}-1}{q^{t}-1}\right)^{a}\left(\frac{q^{t}-q^{t-r}}{q^{t}-1}\right)^{k-\left(q^{t}-1\right)-a}\cdot\\
&{}\qquad\qquad\left(1-\frac{1}{q^{t-r}}\right)^{a+q^{t-r}-1}\\
 & =\left(\frac{\left(q^{t-r}-1\right)^{2}}{q^{t-r}\left(q^{t}-1\right)}+\frac{q^{t}-q^{t-r}}{q^{t}-1}\right)^{k-\left(q^{t}-1\right)}\left(\frac{q^{t-r}-1}{q^{t-r}}\right)^{q^{t-r}-1}\\
 & =\left(1-\frac{1-q^{r-t}}{q^{t}-1}\right)^{k}\left(1-\frac{1-q^{r-t}}{q^{t}-1}\right)^{-\left(q^{t}-1\right)}\left(1-\frac{1}{q^{t-r}}\right)^{q^{t-r}-1}\\
 & \le \left(1-\frac{1-q^{r-t}}{q^{t}-1}\right)^{k},
\end{align*}
where the last inequality is valid since
\[\left(1-\frac{1}{q^{t-r}}\right)\le\left(1-\frac{1}{q^t-1}+\frac{1}{q^{t-r}(q^t-1)}\right)^{\frac{q^t-1}{q^{t-r}-1}},\]
which follows from the fact that the function $(1-kx)^{1/x}$ is monotone decreasing on the interval $[0,1]$ for fixed $k\in(0,1)$.

Repeating the process for $p_{2}$, let $\mathcal{B}$ be the number
of rows in $\Delta$ that do not begin with the last $r$ characters
of $\gamma$ in some fixed fashion; $\mathcal{B}$ is at least $q^{t-r}$. Then,

\[
\p(\mathcal{B}=b+q^{t-r})=\binom{k-\left(q^{t}-1\right)}{b}\left(\frac{q^{t-r}}{q^{t}-1}\right)^{b}\left(\frac{q^{t}-1-q^{t-r}}{q^{t}-1}\right)^{k-(q^{t}-1)-b},
\]
and by the same reasoning as before,

\begin{align*}
p_{2} & =\sum_{b=0}^{k-\left(q^{t}-1\right)}\binom{k-\left(q^{t}-1\right)}{b}\left(\frac{q^{t-r}}{q^{t}-1}\right)^{b}\left(\frac{q^{t}-1-q^{t-r}}{q^{t}-1}\right)^{k-\left(q^{t}-1\right)-b}\cdot\\
&\qquad\qquad\left(1-\frac{1}{q^{t-r}}\right)^{b+q^{t-r}}\\
 & =\left(1-\frac{1}{q^{t}-1}\right)^{k}\left(1-\frac{1}{q^{t}-1}\right)^{-\left(q^{t}-1\right)}\left(1-\frac{1}{q^{t-r}}\right)^{q^{t-r}}\\
 & \le\left(1-\frac{1}{q^{t}-1}\right)^{k}.
\end{align*}
Therefore,

\begin{align}
\p\left(I_{\Delta}=1|A_{1}\right) & \le q^{t-r}p_{1}+\left(q^{t}-q^{t-r}\right)p_{2}\nonumber\\
 & \le q^{t-r}\left(1-\frac{1-q^{r-t}}{q^{t}-1}\right)^{k}+\left(q^{t}-q^{t-r}\right)\left(1-\frac{1}{q^{t}-1}\right)^{k}\nonumber\\
 & =q^{t-r}\left(1-\frac{1-q^{r-t}}{q^{t}-1}\right)^{k}\left(1+(q^{r}-1)\left(\frac{q^{t}-2}{q^{r-t}+q^{t}-2}\right)^{k}\right)\nonumber\\
 & =q^{t-r}\left(1-\frac{1-q^{r-t}}{q^{t}-1}\right)^{k}\left\{ 1+o\left(1\right)\right\}\enspace(k\to\infty) ,
\end{align}
and thus by (2) and (3)
\begin{align*}
\p\left(I_{\Gamma}I_{\Delta}=1\right) & \leq\p\left(I_{\Gamma}=1\right)\left[\p\left(I_{\Delta}=1|A_{1}\right)+\left(\frac{q^{t}-1}{2}\right)\left(\frac{q^{t}-2}{q^{t}-1}\right)^{k}\left\{ 1+o\left(1\right)\right\} \right]\\
 & \leq q^{t}\left(1-q^{-t}\right)^{k}\cdot\\
&\qquad\left[q^{t-r}\left(1-\frac{1-q^{r-t}}{q^{t}-1}\right)^{k} +\left(\frac{q^{t}-1}{2}\right)\left(\frac{q^{t}-2}{q^{t}-1}\right)^{k} \right]\{1+o(1)\}\\
 & =\left(q^{2t-r}\left(\frac{q^{t}+q^{r-t}-2}{q^{t}}\right)^{k}+\frac{q^{t}\left(q^{t}-1\right)}{2}\left(\frac{q^{t}-2}{q^{t}}\right)^{k}\right)\left\{ 1+o\left(1\right)\right\} \\
 & =q^{2t-r}\left(\frac{q^{t}+q^{r-t}-2}{q^{t}}\right)^{k}\cdot\\
&\qquad\left(1+\frac{q^{t}-1}{2q^{t-r}}\left(\frac{q^{t}-2}{q^{t}+q^{r-t}-2}\right)^{k}\right)\left\{ 1+o\left(1\right)\right\} \\
 & =q^{2t-r}\left(\frac{q^{t}+q^{r-t}-2}{q^{t}}\right)^{k}\left\{ 1+o\left(1\right)\right\}\enspace(k\to\infty) .
\end{align*}
This proves Lemma 2.5, our main correlation bound.\hfill\end{proof}

Continuing the quest for a lower threshold, we compare the means of $X$, the variable of interest, and $Y$, the maximum number of disjoint collections of unshattered $t$-words.  Denoting the number of overlapping {\it pairs } of unshattered $t$-words by $Z$, we have that
\[Y\le X\le Y+Z,\] so that 
\[\e(X)\le\e(Y)+\e(Z).\]
Now, Fact 10.1 in \cite{mr} is as follows:
\begin{lem}
Let $m$ denote the median of the 1-Lipschitz random variable $Y=Y(Y_1,\ldots,Y_N)$, where the $Y_i$s are independent, and where $Y$ is certifiable using the certification function $f(s)=rs$.  Then
\[
\left|\e\left(Y\right)-m\right|\leq40\sqrt{r\e\left(Y\right)},
\]\end{lem} 
\noindent and so, setting $r=kt$, (1) yields
\begin{eqnarray}
\p(X=0)=\p(Y=0)&\le&2e^{-\frac{m}{4kt}}\nonumber\\
&\le&2e^{-\frac{1}{4kt}\{\e(Y)-40{\sqrt{kt\e(Y)}}\}}\nonumber\\
&\le&2e^{-\frac{1}{4kt}\{\e(X)-\e(Z)-40{\sqrt{kt\e(X)}}\}}.
\end{eqnarray}

The key issue is thus to find conditions under which $\e(Z)\to0$.    By Lemma 2.5, 
\begin{eqnarray}
\e(Z)&=&\sum_{\Gamma\cap\Delta\ne\emptyset}\p(I_\Gamma I_\Delta=1)\nonumber\\
&\le&\sum_{j=1}^{n \choose t}\sum_{r=1}^{t-1}{t\choose r}{{n-t}\choose{t-r}}q^{2t-r}\lr\frac{q^t+q^{r-t}-2}{q^t}\rr^k(1+o(1))\nonumber\\
&\le&K\sum_{r=1}^{t-1}n^{2t-r}\lr\frac{q^t+q^{r-t}-2}{q^t}\rr^k
\end{eqnarray}
for some constant $K=K_{t,q}$.
The $r$th term in (5) tends to zero provided that 
\begin{equation}k\ge\frac{(2t-r)\lg n+\omega_n(1)}{{\lg\left({q^{t}}/({q^{t}+q^{r-t}-2})\right)}},\end{equation}with $\omega_n(1)\to\infty$ being arbitrary.  The next two lemmas enable us to determine when (6) holds for all $r$.
\begin{lem}
The function $f:[2,t-1]\rightarrow\mathbb{R}$ defined by $f(r)=\frac{q^{r-1}-1}{r-1}, q\ge 2$
is monotonically increasing.\end{lem}
\begin{proof}
Define $g(a)=\frac{q^{a}-1}{a}$; for $1\le a\le t-2$.  We have

\[
g'\left(a\right)=\frac{aq^{a}\log\left(q\right)-\left(q^{a}-1\right)}{a^{2}},
\]
and the result follows since $aq^{a}\log q-\left(q^{a}-1\right)=q^{a}\left(a\log q-1\right)+1\geq2\left(\log2-1\right)+1=2\log2-1>0$.\hfill
\end{proof}

\begin{lem}
The constant $\frac{2t-r}{\lg\left({q^{t}}/({q^{t}+q^{r-t}-2})\right)}$ indicated by (6)
is largest when $r=1$.\end{lem}
\begin{proof}
We prove that $\frac{2t-r}{\lg\left({q^{t}}/({q^{t}+q^{r-t}-2})\right)}\le\frac{2t-1}{\lg\left({q^{t}}/({q^{t}+q^{1-t}-2})\right)}$
for integers $q\geq2$, $t\geq3$, and $r\in[2,t-1]$. This occurs if and only if
\[\left(\frac{q^{t}}{q^{t}+q^{1-t}-2}\right)^{2t-r}\le\left(\frac{q^{t}}{q^{t}+q^{r-t}-2}\right)^{2t-1},\]
which is equivalent to
\[\left(1+\frac{q^{1-t}\left(q^{r-1}-1\right)}{q^{t}+q^{1-t}-2}\right)^{2t-1}\left(1+q^{1-2t}-2q^{-t}\right)^{r-1}\le1.\]
Since $1+x\le e^{x}$, it suffices to show that:
\[\exp\lc\left(\frac{\left(2t-1\right)q^{1-t}\left(q^{r-1}-1\right)}{q^{t}+q^{1-t}-2}\right)+\left(r-1\right)\left(q^{1-2t}-2q^{-t}\right)\rc\le1,\]
or
\[\left(2t-1\right)\left(\frac{q^{r-1}-1}{r-1}\right)<2q^{t-1}-1+\frac{4}{q^{t}}-\frac{1}{q^{2t-1}}-\frac{4}{q}.
\]
Since $\frac{4}{q^{t}}>0$ and $1+\frac{1}{q^{2t-1}}+\frac{4}{q}\leq1+\frac{1}{32}+2\le4$,
it then suffices to show that
\[
\left(2t-1\right)\left(\frac{q^{r-1}-1}{r-1}\right)<2q^{t-1}-4,
\]
or, by Lemma 2.7,  that 
\begin{equation}\frac{2t-1}{t-2}(q^{t-2}-1)\le2(q^{t-1}-2).\end{equation}
Now (7) may be verified to be true for $t\ge4; q\ge2$ and for $t=3, q\ge3$.  The remaining case, $t=3; q=2$ can be checked by verifying the statement of Lemma 2.8 directly.
This completes the proof.\hfill\end{proof} 
By (5) and Lemma 2.8, 
\begin{eqnarray}
\e(Z)&\le&K_{t,q}n^{2t-1}\lr\frac{q^t+q^{1-t}-2}{q^t}\rr^k\nonumber\\
&\to&0\quad{\rm if} \enspace k\ge\frac{(2t-1)\lg n}{\lg\lr {q^t}/({q^t+q^{1-t}-2})\rr}(1+o(1));
\end{eqnarray}
the next lemma verifies the rather critical fact that this occurs for $k$'s that are smaller than the lower threshold we plan to exhibit.
\begin{lem}
\[\frac{(2t-1)}{\lg \lr{q^t}/({q^t+q^{1-t}-2)}\rr}<\frac{t}{\lg\lr{q^t}/({q^t-1})\rr}\]
\end{lem}
\begin{proof}
The claim is equivalent to 
\[\lr\frac{q^t}{q^t-1}\rr^{2t-1}<\lr\frac{q^t}{q^t+q^{1-t}-2}\rr^t,\]
or to \[\frac{q^{2t}}{q^{2t}-2q^t+1}\lr\frac{q^t-1}{q^t}\rr^{1/t}<\frac{q^t}{q^t+q^{1-t}-2}.\]  Using the inequalities $1-x\le e^{-x}$ and $e^{-x}\le1/(1+x)$, we see that
\[\lr\frac{q^t-1}{q^t}\rr^{1/t}\le\exp\{-1/(tq^t)\}\le\frac{tq^t}{1+tq^t},\] 
so that it would suffice to show, on simplification, that
\[q^t\lc t(q-1)+2\rc<q^{2t}+1,\]
which is true since $q^t\ge t(q-1)+2; q,t\ge 2$.\hfill
\end{proof}We are now ready to state the main result of this section.
\begin{thm} Consider a $k\times n$ array with entries that are uniformly and independently selected from $\{0,1,\ldots,q-1\}$.  Then if for large enough $A=A_{t,q}$
\[k\le \frac{t\lg n-\Omega(\lg\lg n)}{\lg \frac{q^t}{q^t-1}},\enspace { \it where}\ \Omega(\lg\lg n)\ge A\lg\lg n,\]
then the probability that the array is a $(t,q,n,k,1)$-covering array tends to zero as $n\to\infty$.
\end{thm}
\begin{proof}
Equation (4) reveals that the array will be $t$-covering with probability tending to zero whenever $\e(Z)\to0$ and $\e(X)/kt\to\infty$.  Since 
\[\e(X)\ge{n\choose t}\lr\frac{q^t-1}{q^t}\rr^k\ge\frac{n^t}{t!}\lr\frac{q^t-1}{q^t}\rr^k(1+o(1)),\]
we have that $\e(X)/kt\to\infty$ provided that
\[k\le\frac{t\lg n-\Omega(\lg\lg n)}{\lg\lr\frac{q^t}{q^t-1}\rr}\]
with $\Omega(\lg\lg n)\ge A\lg\lg n$.  Thus  (8) and Lemma 2.9 reveal that $\p(X=0)\to0$ if 
\[\frac{(2t-1)\lg n}{\lg \lr{q^t}/({q^t+q^{1-t}-2})\rr}(1+o(1))\le k\le\frac{t\lg n-\Omega(\lg\lg n)}{\lg\lr{q^t}/({q^t-1})\rr},\enspace \Omega(\lg\lg n)\ge A\lg\lg n. \]
The full conclusion of the theorem follows by monotonicity, in $k$, of $\p(X=0)$.\hfill\end{proof}
To seal the connection between covering arrays and shattering multisets, we restate Theorem 2.10 as follows:
\begin{thm}  Consider $k$ multisets ${\a}=\{A_1,\ldots,A_k\}$ of $[n]$ as follows:

(i) Each element of $[n]$ is represented in $A_i$ at most $q-1$ times; $q\ge 2$, $1\le i\le k$, and 

(ii) The system $\a$ of multisets is randomly generated by choosing the multiplicity of each element $j$ in multiset $A_i$ $(1\le j\le n; 1\le i\le k\}$ independently and uniformly from the set $\{0,1,\ldots,q-1\}$.

Then the collection $\a$ fails to shatter all multisets of $t$ elements, each element repeated $q-1$ times, with high probability if $k\le (t\lg n-\Omega(\lg\lg n))/{\lg \frac{q^t}{q^t-1}}$,

\noindent $\Omega(\lg\lg n)\ge A\lg\lg n$.
\end{thm}
Together with Theorem 2.1, Theorems 2.10 and 2.11 show that the gap between the lower and upper thresholds is rather small; actually this gap arises as an artifact of the Talagrand inequality, and is in fact artificial.  Finally, observe that we have actually uncovered a threshold for the VC dimension of random multiset arrays.  For example, if $q=2$ then sets of size 3 are fully shattered with high probability (w.h.p.) at the level $3\lg n/\lg(8/7)\approx15.57\lg n$.  Thus the VC dimension is 4 or more.  But few sets of size 4 are shattered at this level; they are all shattered w.h.p. when the number of rows are of magnitude $4\lg n/\lg (16/15)\approx42.96\lg n$.  In between these levels, the VC dimension of the set system is thus {\it equal to} 4 w.h.p.

Other proofs of Theorems 2.10, e.g., using the second moment method (see \cite{as}), would have worked just as well as the proof using Talagrand's inequality.   However, the hard correlation analysis of the last few pages would be necessary no matter what method was used.  Moreover, we are forced to use the second moment method in Section 3, when we lose the 1-Lipschitz property.
\section{Shattering Permutations}

For permutations, we use a model analogous to the one used for words
in the previous section. Let $\s$ be a randomly generated set of $k$ permutations $\pi_1,\ldots,\pi_k\in S_n$, with each chosen independently with probability $1/n!$  As before, we can represent $\s$ as an array,
and we will continue to use \textit{rows} to refer to the
elements of $\s$ and \textit{columns} to refer to the positions within
each element of $\s$.

Shattering permutations is conceptually similar to shattering words.
Let $i_{1},i_{2},i_{3}$ be any 3 elements of $\left[n\right]$ and
let $S^{*}$ be the set consisting of the $k$ triples formed
by intersecting the $i_{1}$th, $i_{2}$th, and $i_{3}$th columns with
the $k$ rows of $S$. 
Then, $S$ shatters the triple $(i_1, i_2, i_3)$ (or the {\it positions}  $(i_1, i_2, i_3)$) if $\rho\in S^{*}$ up to order-isomorphism for each $\rho\in S_3$.

In this section, we show that the threshold function for the property
{}``shatters all triples'' under our model is $k_{0}\left(n\right)=\frac{3}{\lg\left(\frac{6}{5}\right)}\lg\left(n\right)$
modulo a small gap. The upper threshold  is obtained via the first moment method as in Section 2.     
We cannot, however, use the same approach as before to prove the lower threshold, as the conditions of Talagrand's inequality are no longer valid, for the following reason:  We generate the $k\times n$ array using $nk$ independent random variables $\{X_{ij}\}_{1\le i\le k; 1\le j\le n}$ each uniformly distributed on $[0,1]$, with the order statistics of the $k$ groups of $n$ consecutive random variables determining the permutation in the corresponding row.  For example for $n=4$ and $k=3$, the sample outcome $X_{13}<X_{12}<X_{14}<X_{11}$;  $X_{21}<X_{24}<X_{23}<X_{22}$; $X_{31}<X_{33}<X_{34}<X_{32}$ produces the permutation array
\begin{eqnarray*}
3\quad2\quad4\quad1\\
1\quad4\quad3\quad2\\
1\quad3\quad4\quad2
\end{eqnarray*}
If $X_{ij}^*\ne X_{ij}$  for just one pair of indices, that could potentially change the permutation in that row drastically, and we can create examples where the maximum number of sets of disjoint unshattered $t$-ples changes from its original value to 0, or go up from a small number to a large number.  

The method of choice for the lower threshold will thus be elementary; we use the second moment method (Chebychev's inequality), which implies that for a non-negative random variable, 
\[\p(X=0)\le\frac{\va(X)}{\e^2(X)}.\]Interestingly, the correlation analysis would have been equally complicated had Talagrand's inequality been applicable.  However, the analysis becomes intractable for higher values of $t$, which is the size of the tuple we wish to shatter, and we thus restrict to $t=3$ in this paper except in the next elementary result, proved originally in \cite{spencer}.
\begin{thm}
Let $n\ge t\ge 2$, and $k\ge\frac{t}{\lg\left(\frac{t!}{t!-1}\right)}\lg\left(n\right)(1+o(1))$.  Consider a rectangular array ${\cal S}$ of $k$ random permutations.  Then all $t$-ples are shattered almost surely by $\s$.\end{thm}
\begin{proof}
Let $X$ be the number of  unshattered
triples. By Markov's Inequality and linearity of expectation, we have:
\[
\p\left(X\geq1\right)\leq\e\left(X\right)\leq{n \choose t}t!\left(\frac{t!-1}{t!}\right)^{k}\leq\frac{n^{t}}{t!}t!\left(\frac{t!-1}{t!}\right)^{k}\to0
\]
provided that $k\ge\frac{t\lg n+\omega_n(1)}{\lg\lr\frac{t!}{t!-1}\rr},$  with $\omega_n(1)\to\infty$ being arbitrary.  This proves the result.\hfill\end{proof}

Now, for the lower threshold. Let $X$ be as before.  We need to compute $\va(X)$ and demonstrate when it is  $o(\e^2(X))$.  Let the indicator random variable $I_\Gamma$ equal $1$ if the $\Gamma$th set of three columns 
is missing at least one 3-permutation, with $I_\Gamma=0$ otherwise.  Set $N={n\choose 3}$.  We have
\begin{eqnarray*}
\va(X)&=&\va\lr\sum_{\Gamma=1}^NI_{\Gamma}\rr\nonumber\\
&=&\sum_{\Gamma=1}^N[\e(I_\Gamma)-\e^2(I_\Gamma)]+2\sum_{\Gamma<\Delta}[\e(I_\Gamma I_\Delta)-\e(I_\D)\e(I_\G)]\nonumber\\
&\le&\e(X)+\sum_{\Gamma,\Delta\in OP}[\p(I_\Gamma I_\Delta=1)-\p(I_\D=1)\p(I_\G=1)],
\end{eqnarray*}
where $\Gamma$ and $\Delta$ range over the set $OP$ of distinct non-disjoint sets of 3 columns.   It follows that
\begin{equation}
\frac{\va(X)}{\e^2(X)}\le\frac{1}{\e(X)}+\frac{\sum_{\Gamma,\Delta\in OP}\p(I_\Gamma I_\Delta=1)}{\e^2(X)}.
\end{equation}Unlike the case of words, the correlation analysis is more complex.  First, given an overlap size we can no longer consider just two cases, since, e.g., two patterns $ijk$ and $abc$ may correspond to overlapping last and first columns in $\Gamma,\Delta$ respectively, but $k$ may equal 3 and $a$ might be 1.  Second, we can no longer assume without loss of generality, as we did with words, that the overlap occurred in the last $r$ columns of $\Gamma$, and that the rest of $\Delta$ was entirely to the right of that overlap.   In other words, correlations depend not just on the magnitude of the overlap, but its nature as well.  With this in mind, 
let $A_{m}$ be the event that exactly $m$ 3-permutations are missing
from $\Gamma, 1\le m\le5$, let $B_{abc}\subseteq A_1$ be the event that the 3-permutation
$abc$ is the only permutation missing from $\Gamma$, and let $C_{ijk}$
be the event that $ijk$ is missing from $\Delta$. Then, we have:
\begin{eqnarray}
\p\left(I_{\Gamma}I_{\Delta}=1\right)&=& \p\left(I_{\Gamma}=1\right)\p\left(I_{\Delta}=1|I_{\Gamma}=1\right)\nonumber\\
 & =&\p\left(I_{\Gamma}=1\right)\p\left(I_{\Delta}=1|A_{1}\cup A_{2}\cup\dotsb\cup A_{5}\right)\nonumber\\
 & =&\p\left(I_{\Gamma}=1\right)\left[\frac{\p\left(I_{\Delta}=1\cap A_{1}\right)+\dotsb+\p\left(I_{\Delta}=1\cap A_{5}\right)}{\p\left(A_{1}\cup A_{2}\cup\dotsc\cup A_{5})\right)}\right]\nonumber\\
 & \leq&\p\left(I_{\Gamma}=1\right)\cdot\nonumber\\
&\quad&\left[\p\left(I_{\Delta}=1|A_{1}\right)+\left(\frac{\p\left(A_{2}\right)+\dotsb+\p\left(A_{5}\right)}{\p\left(A_{1}\right)+\p(A_2)+\dotsb+\p(A_5)}\right)\right]\nonumber\\
 & =&\p\left(I_{\Gamma}=1\right)\ls\p\left(I_{\Delta}=1|A_{1}\right)+\frac{15\cdot(4/6)^k}{6\cdot(5/6)^k-15\cdot(4/6)^k}\rs\nonumber\\
 & =&\p\left(I_{\Gamma}=1\right)
\ls\p\left(I_{\Delta}=1|\bigcup_{ijk\in S_3}B_{ijk}\right)+O((4/5)^k)\rs\nonumber\\
 & =&\p\left(I_{\Gamma}=1\right)\left[\frac{\sum_{ijk\in S_3}\p\left(I_{\Delta}=1\cap B_{ijk}\right)}{\sum_{ijk\in S_3}\p\left(B_{ijk}\right)}+O((4/5)^k)\right]\nonumber\\
 & \leq&\frac{1}{6}\p\left(I_{\Gamma}=1\right)\left(\p\left(I_{\Delta}=1|B_{123}\right)+\ldots+\p\left(I_{\Delta}=1|B_{321}\right)\right)\nonumber\\
&\quad&+\p\left(I_{\Gamma}=1\right)O((4/5)^k))\nonumber\\
 & \leq&
\left(\frac{5}{6}\right)^{k}\left(\p\left(C_{123}|B_{123}\right)+\p(C_{132}|B_{123})+\ldots+\p\left(C_{321}|B_{321}\right)\right)\nonumber\\
&\quad&+O((2/3)^k).
\end{eqnarray}
Accurate estimation of the 36 quantities $\p\left(C_{ijk}|B_{abc}\right)$ will thus be critical.  

Let $\Gamma$ and $\Delta$ be distinct non-disjoint sets of 3 columns.  Let $D_{abc}$ and $F_{ijk}$ be the events that $abc$ appears in a fixed row of $\Gamma$ and $ijk$ appears in the same row of $\Delta$.  Now the probability distribution of the components of $\Gamma$ conditional on $B_{abc}$ can be obtained, as in Section 2, by randomly selecting five rows, placing one pattern other than $abc$ in these rows, and randomly choosing a pattern other than $abc$ to appear in the other rows.  For simplicity, however, we will assume that {\it each} of the $k$ rows in $\Gamma$ is equally likely to be chosen to be one of the non-$abc$ patterns; it can be shown that using this slightly incorrect\footnote{since this yields a non-zero probability of there being four or fewer patterns in $\Gamma$; a completely correct proof would be as in Section 2. } conditional distribution leads to no change in our final conclusion.  
We have, with $A^C$ denoting the complement of $A$, 
\begin{eqnarray}
\p\left(C_{ijk}|B_{abc}\right)&=&\lr\frac{\sum_{uvw\ne abc}\p(D_{uvw}\cap F_{ijk}^C)}{\sum_{uvw\ne abc}\p(D_{uvw})}\rr^k\nonumber\\
&=&\lr{\frac{6}{5}\lr\p(F_{ijk}^C)-\p(D_{abc}\cap F_{ijk}^C)\rr}\rr^k.
\end{eqnarray}
\begin{lem} Assume that $\vert\Gamma\cap\Delta\vert=1$ and let $\gamma,\delta$ refer to the index of the overlapping position in the two sets of columns.  Then we have the probabilities in Table 1.
\begin{table}
\begin{centering}
\begin{tabular}{|c|c|}
\hline 
 $(\gamma,\delta)$&$\p(D_{uvw}\cap F_{ijk}^C\vert B_{abc})$\tabularnewline
\hline 
(1,1)&$\frac{14}{100}$\tabularnewline
(1,2)&$\frac{17}{100}$\tabularnewline
(1,3)&$\frac{19}{100}$\tabularnewline
(2,2)&$\frac{16}{100}$\tabularnewline
(2,3)&$\frac{17}{100}$\tabularnewline
(3,3)&$\frac{14}{100}$\tabularnewline
\hline 
\end{tabular}
\par
\end{centering}
\caption{$\p(D_{uvw}\cap F_{ijk}^C\vert B_{abc})$ for Overlap 1}
\end{table}
\end{lem}
\begin{proof}  It is easier, for any $ijk$ and $uvw$, to calculate the probabilities $\p(D_{uvw}\cap F_{ijk})$ instead of $\p(D_{uvw}\cap F_{ijk}^C)$.  As an example, suppose that the five positions spanned by $(\G,\D)$ are as follows:
\[{{\Gamma}\atop{\Delta}}\quad{{3}\atop{}}\quad{{2}\atop{1}}\quad{{1}\atop{}}\quad{{}\atop{3}}\quad{{}\atop{2}}\]
We call this the $\gamma=2, \delta=1$ case, since the overlapping element in $\Gamma$ is $\gamma=2$ and the overlapping element in $\Delta$ is $\delta=1$.  Now, observe that
\[\p(D_{uvw}\cap F_{ijk}\cap\gamma=g\cap\delta=d)={{g+d-2}\choose{g-1}}\cdot{{6-g-d}\choose{3-g}}/120,\]
since the permutation of length 5 has $g+d-2$ elements smaller than the overlap element, and we need to pick $g-1$ of them that will be in $\Gamma$.  A similar analysis for elements larger than the overlap element provides the other term.  This yields the entries in Table 1; note that the relative positions of the non-overlapping indices amongst the five positions are irrelevant.  \hfill\end{proof}

Returning to the variance in (9), we first address the contribution of the $O(2/3)^k$ term from (10).  Summing this quantity over all choices of $\Gamma,\Delta\in OP$, we see that the net contribution, to $\sum\p(I_\G I_\D=1)$ of terms corresponding to two or more permutations being absent in $\Gamma$ is of order
\[n^5\lr\frac{2}{3}\rr^k\to0\]
if $k\ge(5+o(1))\lg n/\lg (1.5)$, or if $k\ge 8.55\lg n$.

We now turn to the 36 terms in the first part of the last line in (10), each term of which may be bounded using (11) and the minimum value 14/100 from Table 1 as follows:
\begin{eqnarray*}
\p\left(C_{ijk}|B_{abc}\right)&=&
\lr{\frac{6}{5}\lr\p(F_{ijk}^C)-\p(D_{abc}\cap F_{ijk}^C)\rr}\rr^k\\
&\le&\ls\frac{6}{5}\lr\frac{5}{6}-\frac{14}{120}\rr\rs^k\\&=&(0.86)^k,
\end{eqnarray*}  
and thus the contribution of the single overlap case to the correlation in (9) is $O(n^5(5/6\cdot(0.86))^k$, which tends to zero if 
$k\ge10.41\lg n.$

Consider next the two-overlap case.  Equation (11) remains unchanged, but the conditional probabilities $\p(D_{uvw}\cap F_{ijk}^C\vert B_{abc})$ have a denominator of 20.  Given the four entries in $\Gamma\cup\Delta$ in the two-overlap case, the components of the pattern $uvw$ may be chosen in 4 ways, and it remains to calculate how many of these are also consistent with $F_{ijk}$.  There are three cases.  If the components of the two columns in the overlap are {\it identical} as, e.g.,  in 
\[{{3}\atop{}}\quad{{2}\atop{2}}\quad{{1}\atop{1}}\quad{{}\atop{3}}\]
the entries in the four positions may appear in two forms, in this case 3214 or 4213.  If the components of the two columns are {\it consistent} as, e.g., in
\[{{2}\atop{}}\quad{{1}\atop{2}}\quad{{3}\atop{3}}\quad{{}\atop{1}}\]
then there is only one possible arrangement, in this case 3241.  Finally, if the components are {\it inconsistent}, consisting of one monotone increasing and one monotone decreasing pattern, e.g.,
\[{{1}\atop{3}}\quad{{3}\atop{}}\quad{{}\atop{2}}\quad{{2}\atop{1}}\]
then there is clearly no arrangement.  We thus have $\p(D_{uvw}\cap F_{ijk})$ equalling 2/24, 1/24, or 0 in these three cases, or $\p(D_{uvw}\cap F_{ijk}^C)$=2/24, 3/24, or 4/24 respectively.  (10) and (11) thus yield for some constant $C$,
\[
\p(I_\Gamma I_\Delta=1)\le C\cdot\lr\frac{5}{6}\rr^k\cdot\lr\frac{6}{5}\lc\frac{5}{6}-\frac{2}{24}\rc\rr^k=C\cdot(0.75)^k
\]
and thus the contribution of the double overlap case to $\e(Z)$ is of magnitude $n^4(0.75)^k$, which tends to zero if $k\ge 9.64\lg n$.  We are now ready to prove the main result of this section:
\begin{thm}  Suppose we choose $k$ permutations randomly, uniformly, and with replacement from $S_n$, then the probability that they shatter all 3-permutations in any three positions $i_1<i_2<i_3$ tends to zero as $n\to\infty$ provided that $k\le \frac{3\lg n-\omega_n(1)}{\lg 1.2}$ where $\omega_n(1)\to0$ is arbitrary.
\end{thm}
\begin{proof}
From the above analysis, we see that  both components of the right side of (9) tend to zero if $\e(X)\to\infty$ and $k\ge 10.41\lg n$ (actually a weaker condition than this can be derived by using a more careful calculation).  
Now 
\[\e(X)\ge{n\choose 3}\lr6\lr\frac{5}{6}\rr^k-15\lr\frac{4}{6}\rr^k\rr\ge n^3\lr\frac{5}{6}\rr^k(1+o(1))\to\infty\]
if $k$ is as stated.  The full conclusion of the theorem, namely that $\p(X=0)\to0$ for $k\le 10.41\lg n$, follows by monotonicity.
\hfill\end{proof}

\section{Acknowledgments} The research of all three authors was supported by NSF Grant 1004624.  We thank the two anonymous referees whose incisive comments have greatly improved the paper, both in form and in content.

\end{document}